\documentclass[12pt,oneside]{article}
\usepackage[centertags]{amsmath}
\usepackage{amssymb}
\usepackage{hyperref}
\usepackage{amsthm}
\usepackage{bm}
\usepackage{graphicx}% http://ctan.org/pkg/graphicx
\newtheorem{theorem}{Theorem}[section]
\newtheorem{corollary}[theorem]{Corollary}
\newtheorem{lemma}[theorem]{Lemma}

\newtheorem{example}[theorem]{Example}
\newtheorem{definition}[theorem]{Definition}
\newtheorem{remark}[theorem]{Remark}

\numberwithin{equation}{section}
\DeclareMathAlphabet{\mathsc}{OT1}{lmr}{m}{scsl}

\newcommand{\condexp}[1]{\left|}
\newcommand{\Real}{\mathbb R}

\newcommand\mycom[2]{\genfrac{}{}{0pt}{}{#1}{#2}}

\newcommand{\thetabs}{\boldsymbol{\theta}}
\newcommand{\BBs}{\boldsymbol{B}}
\newcommand{\Sbs}{\boldsymbol{S}}
\newcommand{\A}{\mathcal{A}}

\newcommand{\punt}{\boldsymbol{.}}
\newcommand{\En}{\mathbb{E}}
\newcommand{\trasp}{\footnotesize T}

\newcommand{\ibs}{\boldsymbol{i}}

\newcommand{\Nbs}{\boldsymbol{N}}

\newcommand{\mubs}{\boldsymbol{\mu}}
\newcommand{\kappabs}{\boldsymbol{\kappa}}

\newcommand{\jbs}{\boldsymbol{j}}
\newcommand{\ms}{\scriptscriptstyle}

\newcommand{\mbs}{\boldsymbol{m}}
\newcommand{\zbs}{\boldsymbol{z}}
\newcommand{\zerobs}{\boldsymbol{0}}
\newcommand{\unobs}{\boldsymbol{1}}

\newcommand{\ybs}{\boldsymbol{y}}
\newcommand{\xbs}{\boldsymbol{x}}
\newcommand{\mf}{\mathfrak{m}}
\newcommand{\lf}{\mathfrak{l}}
\newcommand{\Xbs}{\boldsymbol{X}}
\newcommand{\Ybs}{\boldsymbol{Y}}

\newcommand{\Tr}{\hbox{\rm Tr}}

\newcommand{\lambdabs}{\boldsymbol{\lambda}}

\newcommand{\C}{\mathcal{C}}

\newcommand\smallN{
    \mathchoice
    {% mode \displaystyle
      \scriptstyle{N}
    }
    {% mode \textstyle
      \scriptstyle{N}
    }
    {% mode \scriptstyle
      \scalebox{0.8}{$\scriptscriptstyle{N}$}
    }
    {% mode \scriptscriptstyle
      \scalebox{0.6}{$\scriptscriptstyle{N}$}
    }
}
\newcommand\smallX{
    \mathchoice
    {% mode \displaystyle
      \scriptstyle{\Xbs}
    }
    {% mode \textstyle
      \scriptstyle{\Xbs}
    }
    {% mode \scriptstyle
      \scalebox{0.8}{$\scriptscriptstyle{\Xbs}$}
    }
    {% mode \scriptscriptstyle
      \scalebox{0.6}{$\scriptscriptstyle{\Xbs}$}
    }
    }
\newcommand\smallY{
    \mathchoice
    {% mode \displaystyle
      \scriptstyle{\Ybs}
    }
    {% mode \textstyle
      \scriptstyle{\Ybs}
    }
    {% mode \scriptstyle
      \scalebox{0.8}{$\scriptscriptstyle{\Ybs}$}
    }
    {% mode \scriptscriptstyle
      \scalebox{0.6}{$\scriptscriptstyle{\Ybs}$}
    }
    }
\newcommand\smalluno{
    \mathchoice
    {% mode \displaystyle
      \scriptstyle{1}
    }
    {% mode \textstyle
      \scriptstyle{1}
    }
    {% mode \scriptstyle
      \scalebox{0.8}{$\scriptscriptstyle{1}$}
    }
    {% mode \scriptscriptstyle
      \scalebox{0.8}{$\scriptscriptstyle{1}$}
    }
}

\newcommand\smalln{
    \mathchoice
    {% mode \displaystyle
      \scriptstyle{n}
    }
    {% mode \textstyle
      \scriptstyle{n}
    }
    {% mode \scriptstyle
      \scalebox{0.8}{$\scriptscriptstyle{n}$}
    }
    {% mode \scriptscriptstyle
      \scalebox{0.8}{$\scriptscriptstyle{n}$}
    }
}
\begin{document}
\title{On multivariable cumulant polynomial sequences with applications}
\author{E. Di Nardo\footnote{Department of Mathematics \lq\lq G.Peano\rq\rq, 
University of Turin, Via Carlo Alberto 10, 10123 Turin, Italia, elvira.dinardo@unito.it}}
\date{ }
\maketitle
\begin{abstract}
A new family of polynomials, called cumulant polynomial sequence, and its extension to the multivariate case is introduced relied on a purely symbolic combinatorial method. The coefficients are cumulants, but depending on what is plugged in the indeterminates, also moment sequences can be recovered. The main tool is a formal generalization of random sums,  when a not necessarily integer-valued multivariate random index is considered. Applications are given within parameter estimations, L\'evy processes and random matrices and, more generally, problems involving multivariate functions. The connection between exponential models and multivariable Sheffer polynomial sequences offers a different viewpoint in employing the method. Some open problems end the paper.
\end{abstract}

\\ \indent

\section{Introduction}
The so-called {\it symbolic moment method} has its roots in the classical umbral calculus developed by Rota and his collaborators since $1964,$ \cite{Rota}. Rewritten in $1994$ (see \cite{Taylor}), what we have called moment symbolic method consists in a calculus on unital number sequences. Its basic device is to represent a sequence $\{1, a_1, a_2, \ldots\}$ with the sequence $\{1, \alpha, \alpha^2, \ldots\}$ of powers of a symbol $\alpha,$ named umbra. The sequence $\{1, a_1, a_2, \ldots\}$ is said to be umbrally represented by the umbra $\alpha.$ 
The main tools of the symbolic moment method are \cite{Bernoulli}: 
\begin{enumerate}
\item[{\it a)}] a polynomial ring ${\mathbb C}[{\cal A}]$ with ${\cal A}=\{ \alpha, \beta, \gamma, \ldots\}$ a set of symbols called umbrae;
\item[{\it b)}] a unital operator $\En: {\mathbb C}[{\cal A}] \rightarrow {\mathbb C},$ called {\it evaluation}, such that
\begin{enumerate}
\item[{\it i)}] $\En[\alpha^i] = a_i$ for non-negative integers $i,$
\item[{\it ii)}] $\En[\alpha^i \beta^j \cdots] = \En[\alpha^i] \En[\beta^j] \cdots$ for distinct umbrae $\alpha, \beta, \ldots$ and non-negative integers $i,j,\ldots$ {\it (uncorrelation property)}.
\end{enumerate}
\end{enumerate} 
The linear operator $\En$ looks like the expectation $E$ of a random variable (r.v.) and $a_i$ is called 
the {\it $i$-th moment\/} of $\alpha.$ This method shares with free probability \cite{Speicher} the employment of moments as a tool to characterize r.v.'s. In particular, a r.v. $X$ is said to be represented by an umbra $\alpha,$ if its sequence of moments is umbrally represented by an umbra $\alpha.$   

One feature of the method is that the same sequence of numbers can be umbrally represented by using two distinct umbrae. When this happens, the umbrae are said {\it similar}. More in details, two umbrae $\alpha$ and $\gamma$ are similar if and only if $\En[\alpha^i] = \En[\gamma^i]$ for all positive integers $i,$ in symbols $\alpha \equiv \gamma.$ 

In the following, we give an example on how to benefit of working with powers instead of indexes.  
\begin{example} \label{1.1} {\rm If $\{a_i\}, \{g_i\}, \ldots, \{b_i\}$ are sequences umbrally represented by the umbrae 
$\alpha, \gamma, \ldots, \zeta$ respectively, then
\begin{equation} 
\En {\underbrace{\left(\alpha + \gamma + \cdots + \zeta\right)}_n}^i = \sum_{k_1 + k_2 + \cdots +k_n = i} {\binom{i}{k_1, k_2, \ldots,k_n}} a_{k_1} g_{k_2} \cdots b_{k_n}
\label{(summ)}
\end{equation}
for all non-negative integers $i.$ If we replace the set $\{\alpha, \gamma, \ldots, \zeta\}$ by a set of $n$ distinct and similar umbrae $\{\alpha, \alpha^{\prime}, \ldots, \alpha^{\prime \prime}\},$ from (\ref{(summ)}) we have 
\begin{equation}
\En[(n \punt \alpha)^i] = \En[{\underbrace{\left(\alpha + \alpha^{\prime} + \cdots + \alpha^{\prime \prime}
\right)}_n}^i]=\sum_{ \lambda \vdash i} (n)_{\lf(\lambda)} d_{\lambda} a_{\lambda}
\label{(summ1)}
\end{equation}
where the summation is over all partitions\footnote[1]{Recall that a partition of an integer $i$ is a sequence 
$\lambda = (\lambda_1, \lambda_2, \ldots, \lambda_t),$ where $\lambda_j$ are weakly decreasing integers and
$\sum_{j=1}^t \lambda_j=i.$ The integers $\lambda_j$ are called parts. The length $\lf(\lambda)$
of $\lambda$ is 
the number of its parts. A different notation is $\lambda = (1^{r_1}, 
2^{r_2}, \ldots),$ where $r_j$ is the number of parts of $\lambda$ equal to $j$ and $r_1 + r_2 + \cdots = 
\lf(\lambda).$ We use the classical notation $\lambda \vdash i$ to denote that $\lambda$ is a partition of $i.$} 
$\lambda = (1^{r_1}, 2^{r_2}, \ldots)$ of the 
integer $i, a_{\lambda} = a^{r_1}_1 a^{r_2}_2 \cdots$ and
$$d_{\lambda} = \frac{i!}{(1!)^{r_1}(2!)^{r_2} \cdots r_1! r_2! \cdots}$$
is the number of $i$-set partitions with block sizes given by the parts of $\lambda.$
The new symbol $n \punt \alpha$ on the left-hand-side of (\ref{(summ1)}) is called the dot-product of $n$ and $\alpha$
and is an example of symbolic device useful to allow generalizations. Indeed in (\ref{(summ1)}), set $\En[(n \punt \alpha)^i] = q_i(n)$ and observe that $q_i(n)$ is a polynomial of degree $i$ in $n.$ Suppose that we replace $n$ by an umbra $\gamma.$ 
The symbol having sequence $\En[q_i(\gamma)]$ as moments is denoted by $\gamma \punt \alpha.$
If $\gamma$ umbrally represents the moment sequence of an integer-valued r.v. $N,$ then 
$\gamma \punt \alpha$ umbrally represents the moments of a random sum $S_{\ms N} = 
X_{\ms 1} + \cdots + X_{\ms N}$ with the independent and identically distributed (i.i.d.) r.v.'s 
$\{X_{\ms i}\}$ umbrally represented by $\alpha.$ Then the umbra $\gamma \punt \alpha$ 
is a symbolic device to manage a sum of $\gamma$ times the umbra $\alpha.$}
\end{example}
Despite the considerable development of the method from its first version \cite{Taylor} and the various applications within probability and statistics (see \cite{DiNardoreview} for a review), the aim of this paper is to re-formulate this symbolic method in terms of cumulant sequences. The reasons of this choice are twofold. First, cumulant sequences have nicer properties compared with moment sequences. In particular, if $\{c_i(X)\}$ is the cumulant sequence of a r.v. $X,$ then the following properties hold for all non-negative integers $i$:  (Homogeneity) $c_i(a X) = a^i c_i (X)$ for $a \in {\mathbb C},$ (Semi-invariance) $c_1(X+a)=a + c_1(X), c_i(X+a) = c_i(X)$ for $i \geq 2,$ (Additivity) $c_i(X_1+X_2) = c_i(X_1) + c_i(X_2),$ if $X_1$ and $X_2$ are independent r.v.'s. Moreover they represent a nimble tool to deal with random sums $S_{\ms N} = X_{\ms 1} + \cdots + X_{\ms N}.$ Indeed if $K_{\ms N}(z)$ is the cumulant generating function (cgf) of $N$ and $K_{\ms X}(z)$ is the cgf of $X_i,$ assumed to be convergent in some open set, then
\begin{equation}
K_{S_{\ms N}}(z)=K_{\ms N}(K_{\ms X}(z))
\label{randomsum}
\end{equation}
is the cgf of $S_{\ms N}.$ Although, in \cite{Berndt} Sturmfels and Zwiernik have underlined that \lq\lq the umbral calculus is an approach to combinatorial sequences using cumulants,\rq\rq up to now, by using both the operator theory  and the symbolic method, the theoretical approach has been focused on moments. Indeed, when in $1994$ Rota decided to turn upside down the umbral calculus, outlining what he considered to be the correct syntax for this matter, he used the Laplace transform. In this paper essentially we propose to employ the logarithm of Laplace transform not only for the properties of cumulants but also because, according to (\ref{randomsum}), to work with compositions of power series corresponds to work with cumulant sequences. Recall that generating functions (gf's) are employed in the so-called symbolic combinatorics. Symbolic combinatorics is a unified algebraic theory dedicated to functional relations between counting gf's, employed in place of more traditional recurrences \cite{FlS}. Gf's are formal power series like $\sum_{i \geq 0} a_i z^i \in {\mathbb C}[[z]]$ and operations on gf's exactly encode operations on counting sequences through coefficients.  Moreover, differently from symbolic combinatorics, the symbolic method works on polynomial sequences and operations on gf's encode operations on polynomial sequences in such a way that new sequences are generated by a suitable replacement of the indeterminates, as shown in Example \ref{1.1}. This replacement allows us to speed up many of the computations usually employed in statistics and involving polynomial sequences. In this paper, we show that further simplifications are obtained if instead of working with polynomials whose coefficients are moments, as in Example \ref{1.1}, we deal with polynomials whose coefficients are cumulants and for this reason called cumulant polynomial sequences. Despite the name, they are moments of stochastic processes with independent and stationary increments \footnote{A continuous-time stochastic process is a family of r.v.'s $\{X_t\}$ indexed by a time $t \geq 0.$ The increments are the differences $X_s-X_t$ between its values at different times $t < s.$ The increments of $X_t$ are said to be independent if $X_s-X_t$ and $X_u-X_v$ are independent r.v.'s, whenever the time intervals $(t,s)$ and $(v,u)$ do not 
overlap and for any non-overlapping time intervals. The increments are said to be stationary if their probability distribution depends only on the time interval length $s-t.$} but depending on what is plugged in place of indeterminates, cumulants can be recovered too.
The usefulness of this device is showed along the paper with examples and through different applications involving both stochastic processes and multivariate statistics.

For shortness, we refer to the multivariate setting of cumulants with exponential gf's. A sequence $\{c_{\ms \ibs}\}$ is the (multivariate) cumulant sequence of $\{m_{\ms \ibs}\}$ if
\begin{equation}
\sum_{\ibs \geq 0} m_{\ibs} \frac{\zbs^{\ibs}}{\ibs!} = \exp \left( \sum_{\ibs>0} c_{\ibs}
\frac{\zbs^{\ibs}}{\ibs!}\right).
\label{(mgfcum)}
\end{equation}
Equation (\ref{(mgfcum)}) is well defined in the ring\footnote[2]{The ring of formal power series ${\mathbb C}[[\zbs]]$ is the set ${\mathbb C}^{\mathbb N}$ of infinite sequence of elements of ${\mathbb C},$ written as infinite series $\sum_{\ibs \geq 0}a_{\ibs} \zbs^{\ibs},$ where $\zbs$ is a formal indeterminate. Formal power series extend the usual operations on polynomials. Any process on series involving each coefficient with only finitely many operations is well-defined.} of formal power series ${\mathbb C}[[\zbs]]$ (cf. \cite{FlS}). For $d=1,
h(t)=\exp(t)$ and $f(\zbs) = \sum_{\ibs > 0} c_{\ibs} \zbs^{\ibs} / \ibs!,$ equation (\ref{(mgfcum)}) is a particular case of
composition of general formal power series 
\begin{equation}
F(\zbs) = h(t_{\ms \smalluno}, \ldots, t_d)|_{t_i=f_{\ms i}(\zbs) \atop \! i=1, \ldots, d} = h[f_{\ms \smalluno}(\zbs), \ldots, f_{\ms d}(\zbs)]. 
\label{(faamult)}
\end{equation}
Note that to allow computations $f_{\ms i}(\zerobs)=0$ for $i=1, \ldots, d$ that is $f_{\ms i}(\zbs)$ needs to be a so-called delta series. The multivariate Fa\`a di Bruno's formula returns the coefficients of (\ref{(faamult)}). Differently from \cite{Faa}, in the following, we consider both $h$ and $\{f_{\ms \smalluno}, \ldots, f_{\ms d}\}$ to be cgf's.  This corresponds to work with a random sum\footnote[3]{As we work with different types of random sums, univariate r.v.'s indexed by multivariate r.v.'s, multivariate r.v.'s indexed by univariate r.v.'s and multivariate r.v.'s indexed by multivariate r.v.'s, then we use the bold letter to help in distinguishing among different cases.} $\Sbs_{\Nbs} = \Xbs_1 + \cdots  + \Xbs_{\Nbs}$ of independent random vectors i.d. to $\Xbs,$ indexed by a multivariate non-negative r.v. $\Nbs.$ The new symbolic calculus we are going to propose allows us to replace the index $\Nbs,$ with a $d$-tuple of symbols representing a sequence of cumulants not necessarily corresponding to an integer-valued random vector.

The rest of the paper is organized as follows. In Section $2$ the symbolic method is introduced in terms of cumulants. Section $3$ is devoted to cumulant polynomial sequences and their applications to L\'evy processes and multivariate statistics. Section $4$ introduces some generalization of cumulant polynomial sequences to the multivariate case. Applications to random matrices and exponential models are given. A special attention is devoted to connections with symmetric functions and their applications to simple random sampling. All the matter is dealing with the syntax of random vectors: the symbolic calculus is employed as a tool to lighting proofs and computations when necessary.  We focus on some open problems at the end of the paper.
%---------------------------------------------------------------------------------------------------------
\section{The cumulant symbolic method}
%---------------------------------------------------------------------------------------------------------
To introduce the multivariate version of the symbolic method in terms of cumulants we need to focus our attention on additivity property,
formulating the uncorrelation property {\it ii)} of the linear operator $\En$ in a different way. 

Let us consider $d$ umbral polynomials $\{\kappa_{\ms 1}, \ldots, \kappa_{\ms d}\} \in {\mathbb C}[{\mathcal A}].$ We define the support of an umbral polynomial as the set of all umbrae occurring in it and we assume $\{\kappa_{\ms 1}, \ldots, \kappa_{\ms d}\}$ have support not necessarily disjoint. For $d$-tuples $\kappabs = (\kappa_{\smalluno}, \ldots, \kappa_{\ms d}),$ the support is the union of the supports of $\{\kappa_{\smalluno}, \ldots, \kappa_{\ms d}\}.$
\begin{definition} \label{2.1}
Two $d$-tuples $\kappabs$ and $\tilde{\kappabs}$ are said to be uncorrelated if they have disjoint supports.
\end{definition}
Now assume $\{c_{\ibs}\} \in  {\mathbb C}$ with $\ibs=(i_{\smalluno}, \ldots, i_{\ms d}) \in {\mathbb N}_{\ms 0}^d$ a multi-index of non-negative integers. Define an unital linear operator $\En: {\mathbb C}[{\mathcal A}] \rightarrow {\mathbb C}$ such that
\begin{description}
\item[{\it a.1)}] $\En \left[ \kappa_{\ms 1}^{i_{\smalluno}} \cdots \kappa_{\ms d}^{i_{\ms d}} \right] = \En \left[ \kappabs^{\ms \ibs} 
\right] = c_{\ibs},$ for $d$-tuples $\kappabs = (\kappa_{\smalluno}, \ldots, \kappa_{\ms d});$
\item[{\it b.1)}] $\En\left[\kappabs^{\ibs} \tilde{\kappabs}^{\jbs} \cdots \right] = 0$ for all $\ibs,\jbs, \ldots \in {\mathbb N}_{\ms 0}^d$ if $\kappabs, \tilde{\kappabs}, \ldots$ are uncorrelated $d$-tuples.  
\end{description}
Assume $\En \left[\kappabs^{\ms \zerobs}\right]=1.$ From {\it a.1)}, we say that the sequence $\{c_{\ibs}\}$ is umbrally represented by the $d$-tuple $\kappabs$ and $c_{\ibs}$ is the $\ibs$-th cumulant of $\kappabs.$  According to Definition \ref{2.1}, let us observe that $\En \left[ \kappabs^{\ms \ibs} \right] = 0$ for some $\ibs,$ if the elements in the $d$-tuple $(\kappa_{\ms 1}, \ldots, \kappa_{\ms d})$ can be split in (at least) two tuples (of lenght less than $d$)  such that any element of the first tuple has disjoint support with any element of the second tuple.
\begin{lemma}
If $\kappabs$ and $\tilde{\kappabs}$ are uncorrelated $d$-tuples, then
\begin{equation}
\En[(\kappabs+\tilde{\kappabs})^{\ibs}] = \En[\kappabs^{\ibs}] + \En[\tilde{\kappabs}^{\ibs}] \quad \hbox{for all} \,\, \ibs \in {\mathbb N}_{\ms 0}^d.
\label{(sum)}
\end{equation}
\end{lemma}
\begin{definition}
If $\En \left[ \kappabs^{\ms \ibs} \right] = \En \left[ (\kappabs^{\prime})^{\ms \ibs} \right]$ for all ${\ibs} \in {\mathbb N}_{\ms 0}^d,$ then $\kappabs$ and $\kappabs^{\prime}$ are said to be similar, in symbols $\kappabs \equiv \kappabs^{\prime}.$ 
\end{definition}
Note that there always exists two (or more) $d$-tuples of similar umbrae $\kappabs$ and $\kappabs^{\prime},$ having $\{c_{\ibs}\}$ as the cumulant sequence.

\begin{definition} \label{(cgfdef)}
The gf of $\kappabs$ is $f(\kappabs, \zbs) = \sum_{\ibs \geq \zerobs} c_{\ibs} \frac{\zbs^{\ibs}}{\ibs!} \in {\mathbb C}[[\zbs]].$
\end{definition}
Note that the gf of $\kappabs$ is such that $f(\kappabs, \zerobs)=1$ differently from the cgf of a random vector such that $K(\zerobs)=0.$ Moreover, if $\kappabs$ and $\tilde{\kappabs}$ are uncorrelated $d$-tuples then $f(\kappabs + \tilde{\kappabs}, \zbs)=f(\kappabs, \zbs) + f(\kappabs^{\prime}, \zbs)$ and $\kappabs \equiv \tilde{\kappabs}$ if and only if $f(\kappabs, \zbs)=f(\kappabs^{\prime}, \zbs).$

\begin{example}[Multivariate Gaussian random vector] \label{Gauss}
{\rm Assume $\Xbs$ a multivariate Gaussian random vector with mean $\mbs$ and full rank covariance matrix $\Sigma,$ that is $\Xbs \sim N({\mathbf m}, \Sigma).$
The symbolic counterpart of $\Xbs$ is the $d$-tuple $\kappabs_{\smallX}$ with gf $f(\kappabs_{\smallX}, \zbs) = 1 + \langle \mbs, \zbs \rangle + \frac{1}{2} \langle \zbs, \zbs \Sigma \rangle,$ with $\langle \cdot , \cdot  \rangle$ the usual scalar product of vectors.}
\end{example}
%------------------------------------------------------------------------
\subsection{Generalized umbral sum}
%------------------------------------------------------------------------
%
Convolutions of independent r.v's correspond to a summation of cumulant sequences; random summations of independent r.v.'s corresponds to the composition of cumulant sequences, and generalize convolutions. In this section we generalize random summations to the new umbral setting.

We first recall a combinatorial tool introduced in \cite{Faa}, paralleling the notion of integer partition for multi-indexes. A partition $\lambdabs$ of a multi-index $\ibs$ is a matrix $\lambdabs = (\lambda_{\ms rs}) \vdash \ibs$ of non-negative integers and with no zero columns, in lexicographic order, such that $\lambda_{\ms r1}+\lambda_{\ms r2}+\cdots+\lambda_{\ms rk}=i_r$ for $r=1,2,\ldots,d.$ As for integer partitions, the notation $\lambdabs = (\lambdabs_{\ms 1}^{\ms r_{\ms 1}}, \lambdabs_{\ms 2}^{r_{\ms 2}}, \ldots)$ denotes a matrix $\lambdabs$ with $r_{\ms 1}$ columns equal to $\lambdabs_{\ms 1}$, $r_{\ms 2}$ columns equal to $\lambdabs_{\ms 2}$ and so on, with $\lambdabs_{\ms 1} < \lambdabs_{\ms 2} < \cdots$. The integer $r_i$ is called multiplicity of $\lambdabs_{\ms i}.$ Set $\mf(\lambdabs)=(r_{\ms 1}, r_{\ms 2},\ldots).$ The number of columns of $\lambdabs$ is denoted by $\lf(\lambdabs).$
For example if $\ibs = (2,1)$ then
\begin{equation}
\lambdabs = \left\{ \binom{2}{1}, \binom{0 \,\, 2}{1 \,\, 0}, \binom{1 \,\, 1}{0 \,\, 1}, \binom{0 \,\, 1 \,\, 1}{1 \,\, 0 \,\, 0} \right\}.
\label{(lambda1)}
\end{equation}
\begin{remark}{\rm A partition $\lambdabs \vdash \ibs$ is a different way to encode multiset partitions with multiplicities $\ibs.$ 
For example, if $M = \{ \underbrace{\mu_1, \mu_1}_2, \underbrace{\mu_2}_1\}$ with $\ibs=(2,1)$ then 
$$\Bigl\{\{\mu_1, \mu_1, \mu_2\}\Bigr\}, \Bigl\{\{\mu_1, \mu_1\}, \{\mu_2\}\Bigr\}, \Bigl\{\{\mu_1, \mu_2\}, \{\mu_1\}\Bigr\}, \Bigl\{\{\mu_1\}, \{\mu_1\}, \{\mu_2\}\Bigr\}$$
are partitions of $M$ whose multiplicities correspond to the different partitions  $\lambdabs$ of $\ibs$ in (\ref{(lambda1)}).}
\end{remark}

Assume $\Sbs_N = \Xbs_1 + \cdots + \Xbs_N$ a sum of indipendent random vectors i.d. to $\Xbs,$ indexed by a non-negative integer-valued r.v. $N.$ In the following the cumulant sequence of $\Xbs$ is denoted by $\{c_{\ibs}\},$ if there exists. Otherwise we refer to the umbra $\kappabs_{\smallX}$ representing the sequence  $\{c_{\ibs}\}$ through the evaluation $\En.$
The same holds for the index $N.$
\begin{lemma}\label{multcomp}
If $\{h_{\ms \ibs}\}$ is the cumulant sequence of a random sum $\Sbs_{\ms N},$ $\{g_{\ms k}\}$ is the cumulant sequence of $N$
and $\{c_{\ibs}\}$ is the cumulant sequence\footnote[4]{When necessary, we explicitly write $c_{\ibs}=c_{\ibs}(\Xbs).$} of $\Xbs$ then
\begin{equation}
h_{\ms \ibs} = \ibs! \sum_{\lambdabs \vdash \ibs} \frac{g_{\ms \lf(\lambdabs)}}{\mf(\lambdabs)! \lambdabs!} \, \prod_{j} c_{\lambdabs_j}^{r_j},
\label{(detsum1)}
\end{equation}
where the summation is over all partitions $\lambdabs$ of the multi-index $\ibs.$
\end{lemma}
\begin{proof}
Assume $f(\kappabs_{\ms \Sbs_{\scriptscriptstyle\smallN}},\zbs), f(\kappa_{\scriptstyle\smallN}, z)$ and $f(\kappabs_{\smallX}, \zbs)$ the gf's of umbral monomials corresponding to $\Sbs_{\ms N}, N$ and $\Xbs$ respectively. According to (\ref{(faamult)}), in the composition of $f(\kappa_{\scriptstyle\smallN}, z)$ with $f(\kappabs_{\smallX}, \zbs)$ the inner power series has to be a delta series, so that
\begin{eqnarray}
f(\kappabs_{\ms \Sbs_{\scriptscriptstyle\smallN}},\zbs) & = & f[\kappa_{\scriptstyle\smallN}, f(\kappabs_{\smallX}, \zbs)-1] = 1 + \sum_{j \geq 1} \frac{g_j}{j!} \left[ f(\kappabs_{\smallX},\zbs) - 1 \right]^j \label{(comp)} \\
& = & 1 + \sum_{j \geq 1} \frac{g_j}{j!} \left\{ \sum_{\ibs > 0} \left[ \sum_{\lambdabs \vdash \ibs \atop {\lf}(\lambdabs) = j}
\binom{\ibs}{\lambdabs} c_{\lambdabs} \right] \frac{\zbs^{\ibs}}{\ibs!} \right\}.
\label{(cumsumrandom)}
\end{eqnarray}
The result follows by suitable rearranging the terms in  (\ref{(cumsumrandom)}).
\end{proof}
An algorithm to implement in {\tt Maple} equation (\ref{(detsum1)}) is available in \cite{UMFB}. In the following, instead of using $\kappabs_{\ms \Sbs_{\scriptscriptstyle\smallN}},$ we introduce the symbol  $\kappa_{{\scriptstyle\smallN}} \punt \kappabs_{\smallX}.$ 
\begin{definition} \label{def1}
The umbra $\kappa_{{\scriptstyle\smallN}} \punt \kappabs_{\smallX}$ is called umbral sum.
\end{definition}
If $P(N=n)=1,$ then $\kappa_{{\scriptstyle n}}$ represents the sequence $\{1,n,0,0,\ldots\}$ and $\kappa_{{\scriptstyle n}} \punt \kappabs_{\smallX}$ represents the cumulant sequence of $S_n.$ The umbral sum is not commutative and the left-hand-side unity is the umbra $\kappa_{u}$ representing the sequence $\{1,1,0,0,\ldots\}.$ Indeed from (\ref{(comp)}) we have
$\kappa_{u} \punt \kappabs_{\ms {\smallX}} \equiv \kappabs_{\ms {\smallX}}.$

Following Example \ref{1.1}, the advantage of the symbolic method consists in replacing the r.v. $N$ by a suitable umbra, allowing us to work with random sums not necessarily indexed by a non-negative integer-valued r.v.
\begin{definition} \label{def1}
The umbra $\kappa_{\ms \alpha} \punt \kappabs_{\smallX}$ is called generalized umbral sum.
\end{definition}
The umbra $\kappa_{\ms \alpha} \punt \kappabs_{\smallX}$ is a generalization of a random sum since its gf is a composition of gf's, as next proposition shows
\begin{lemma} \label{7}
$f(\kappa_{\ms \alpha} \punt \kappabs_{\ms {\smallX}},\zbs) = f\left[\kappa_{\ms \alpha}, f(\kappabs_{\ms {\smallX}},\zbs)-1\right] = \En \left( \exp \left\{\kappa_{\ms \alpha} [f(\kappabs_{{\smallX}}, \zbs) - 1] \right\}\right).$
\end{lemma}
\begin{proof}
The result follows from (\ref{(cumsumrandom)}) as $g_j= \En(\kappa_{\ms \alpha}^{\ms j})$ for $j \geq 1.$
\end{proof}
Note that the last equality in Lemma \ref{7} allows us to deal with the gf of a generalized umbral sum as it was the exponential of a formal power series.
Indeed if a conditional evaluation is considered dealing with $\kappa_{\ms \alpha}$ as it was a constant, the
generalized umbral sum $\kappa_{\ms \alpha} \punt \kappabs_{\ms {\smallX}}$ has $\exp \left\{\kappa_{\ms \alpha} [f(\kappabs_{{\smallX}}, \zbs) - 1] \right\}$ as gf.
\begin{corollary} \label{1}
The auxiliary umbra $\kappa_{\ms n} \punt \kappabs_{\smallX}$ represents the sequence $\{n c_{\ibs}\}$ with $f(\kappa_{\ms n} \punt \kappabs_{\smallX},\zbs) = 1 + n [f(\kappabs_{\smallX}, \zbs) - 1].$
\end{corollary}
Assume to replace ${\mathbb C}[\A]$ with ${\mathbb C}[y][\A].$ Thanks to this device, sequences of polynomials $\{p_n(y)\}$ can be considered as cumulants of what we call {\it polynomial umbrae}. As example, let us consider the umbra $\kappa_{\ms y u},$ denoted by $\kappa_{\ms y}$ to simplify notation. This polynomial umbra represents the sequence $\{1,y,0,0,\ldots\}$ and its r.v. counterpart is $Y$ such that $P(Y=y)=1.$
\begin{corollary} \label{2}
The auxiliary umbra $\kappa_{\ms y} \punt \kappabs_{\smallX}$ represents the sequence $\{y c_{\ibs}\}$ with $f(\kappa_{\ms y}
\punt \kappabs_{\smallX},\zbs) = 1 + y \, [f(\kappabs_{\smallX}, \zbs) - 1].$
\end{corollary}
More general fields ${\mathbb C}[y_1, \ldots, y_n]$ can be considered, as done in the Section $4.$
%----------------------------------------------------------------------------
\section{Cumulant polynomial sequences}
%----------------------------------------------------------------------------
The sequence $\{y c_{\ibs}\} \in {\mathbb C}[y]$ in Corollary \ref{2} is an example of polynomial sequence that can be represented by a polynomial umbra. Polynomial sequences having a more general structure are represented as follows.
\begin{definition} \label{BCP}
If $\Xbs$ is a random vector with $\ibs$-th cumulant $c_{\ibs},$ then the $\ibs$-th cumulant polynomial of $\Xbs$ is
\begin{equation}
\C_{\ibs, \smallX}(y) = \ibs! \, \sum_{\lambdabs \vdash \ibs} \frac{y^{\ms \lf(\lambdabs)}}{\mf(\lambdabs)! \lambdabs!}
 \prod_{j} c_{\lambdabs_j}^{r_j}.
\label{(detsum)}
\end{equation}
\end{definition}
If $\ibs=(2,1)$ from (\ref{(lambda1)}), we have $\C_{\ms {2,1}}(y) = y^3 c_{\ms {0,1}} c^2_{\ms {1,0}} + 2 \, y^2 \, c_{\ms {1,0}} c_{\ms {1,1}} + y^2 c_{\ms {0,1}} c_{\ms {2,0}} + y c_{\ms {2,1}}.$ The name {\it cumulant polynomial} depends on the coefficients involving a sequence of cumulants. In particular  $\C_{\ibs, \smallX}(0)=0,$ since $\lf(\lambdabs) >0$ for all partition $\lambdabs$ of a multi-index.   Cumulant polynomials $\C_{\ibs, \smallX}(y)$ are a generalization of complete Bell (exponential) polynomials \cite{Bernoulli}.

\begin{lemma} \label{3.3}
If $K_{\Xbs}(\zbs)$ is the cgf of $\Xbs,$ then $\sum_{\ibs \geq 0} \C_{\ibs, \smallX}(y) \frac{\zbs^{\ibs}}{\ibs!} = \exp\left\{ y K_{\Xbs}(\zbs) \right\}.$
\end{lemma}
\begin{proof}
The result follows from Lemma \ref{multcomp}, as $\C_{\ibs, \smallX}(y)$ are coefficients of the composition of
$e^{yz}$ and $K_{\Xbs}(\zbs).$
\end{proof}
\begin{remark}
{\rm If $\Xbs$ does not admit cgf, then we refer to formal cumulants  (\ref{(mgfcum)}), and instead of writing $\C_{\ibs, \smallX}(y),$ we write $\C_{\ibs, \mubs}(y)$ assuming the sequence $\{c_{\ibs}\}$ represented by $\kappabs_{\mubs} \equiv (\kappa_{\smalluno}, \ldots, \kappa_{d}).$ In Lemma \ref{3.3}, instead of using $K_{\Xbs}(\zbs)$ we refer to the formal power series $f(\kappabs_{\mubs},\zbs)-1.$}
\end{remark}
Next theorem shows that $\C_{\ibs, \smallX}(y)$ are moments of special multivariate stochastic processes. 
\begin{theorem} \label{3.2} If $\{\Xbs_y\}_{y \geq 0}$ is a multivariate stochastic process with independent and stationary increments, then
$\C_{\ibs, \smallX}(y) = E[\Xbs^{\ibs}_y].$
\end{theorem}
\begin{proof}
For a multivariate stochastic process with independent and stationary increments \cite{Sato}, the moment generating function (mgf) $\phi_y(\zbs)$ is such that $\phi_y(\zbs) = \exp \{ y [K_{\Xbs_1}(\zbs)] \},$ with $K_{\Xbs_1}(\zbs)$ the cgf of $\Xbs_1.$ The result follows from
Definition \ref{BCP} and Lemma \ref{3.3}. 
\end{proof}
\begin{corollary} \label{3.5} If $\{\Xbs_1, \ldots, \Xbs_n\}$ are i.i.d. random vectors, then $\C_{\ibs, \smallX}(n) = E[(\Xbs_1 + \cdots + \Xbs_n)^{\ibs}].$  
\end{corollary}
\begin{example}[Merton's Jump Diffusion Model] {\rm Let us consider the multivariate version of Merton's model \cite{Merton}, that is
$\Xbs_t = \mbs t + \BBs_t + \sum_{j = 1}^{N_t} \Ybs_j$ with $\mbs \in \Real^d, \BBs_t$ a multivariate Brownian motion with covariance rate
$\Sigma t, N_t$ a Poisson process with intensity rate $\lambda$ and $\{\Ybs_j\}$ a sequence of i.i.d. random vectors with multivariate
Gaussian distribution $\Ybs \sim N(\tilde{\mbs}, \tilde{\Sigma}).$ By using the L\'evy-Khintchine formula \cite{Sato}, the mgf of $\Xbs_t$ is
$\exp \{ t [K_{\Xbs_1}(\zbs)]\}$ with
\begin{equation}
K_{\Xbs_1}(\zbs) = \langle \mbs, \zbs \rangle + \frac{1}{2} \langle \zbs, \zbs \Sigma t \rangle + \lambda \int_{\Real^d - \{0\}} \left( e^{\langle \xbs, \zbs \rangle} - 1 \right) \nu({\rm d} \xbs),
\label{(cum1)}
\end{equation}
with $\nu$ the multivariate (L\'evy) measure $N(\tilde{\mbs}, \tilde{\Sigma}).$ By considering the Taylor expansion of the integrand function in (\ref{(cum1)}) we have
\begin{equation}
K_{\Xbs_1}(\zbs) = \langle \mbs + \tilde{\mbs}, \zbs \rangle + \frac{1}{2} \langle \zbs, \zbs ( \Sigma t + \tilde{\Sigma} ) \rangle + \lambda \sum_{\ibs \geq 3 } \frac{\zbs^{\ibs}}{\ibs!} E \left(\Ybs^{\ibs} \right).
\label{(cum2)}
\end{equation}
From Theorem \ref{3.2}, moments of $\En[\Xbs^{\ibs}_t]=\C_{\ibs, \smallX}(t) $ are cumulant polynomials with $\{c_{\ibs}\}$ coefficients of the formal power series (\ref{(cum2)}).}
\end{example}
More in general, depending on what is plugged in the indeterminate $y,$ cumulant polynomial sequences give also cumulants. So they play a double role as the following theorem proves for random sums.

\begin{theorem} \label{cor1bis} If $\Sbs_N$ is a random sum and $c_{\ibs}=c_{\ibs}(\Sbs_N)$ is its $\ibs$-th cumulant, then
$$\En[\C_{\ibs,\smallX}(\kappa_N)] = c_{\ibs}(\Sbs_N) \quad \hbox{and} \quad E[\C_{\ibs,\smallX}(N)] = E[\Sbs^{\ibs}_N].$$
\end{theorem}

\begin{proof}
The first result follows from Lemma \ref{multcomp}. The second result follows by observing that $f(N,z)-1$ is the cgf of a
random sum, say $\tilde{N}_{\tiny \hbox{Po}(1)},$ indexed by a Poisson r.v. of parameter $1$ and involving
independent r.v.'s identically distributed to $N.$ So $\En[\kappa_{\tilde{N}}^i] = E[N^i]$ for all non-negative integer
$i.$ Let us consider the random sum $\Sbs_{\tilde{N}_{\tiny \hbox{Po}(1)}} = \Xbs_1 + \cdots + \Xbs_{\tilde{N}_{\tiny
\hbox{Po}(1)}},$ shortly denoted  by $\Sbs_{\tilde{N}}.$ Then we have to prove that $c_{\ibs}(\Sbs_{\tilde{N}}) = E[\Sbs^{\ibs}_N].$ The result is true since $K_{\Sbs_{\tilde{N}}}(z)= K_{\tilde{N}}(K_{\Xbs}(z))=f(N,K_{\Xbs}(z))-1=f(\Sbs_N,z)-1,$ where the last equality follows by observing that $M_{\Sbs_N}(\zbs)=M_N(K_{\Xbs}(\zbs))$ if $M(\cdot)$ denotes the mgf of $\Sbs_N$ and $N$ respectively.
\end{proof}

In order to generalize Theorem \ref{cor1bis} when the r.v. $N$ is replaced by an umbra, we need to introduce a special umbra,
called the {\it Bell} umbra denoted by the symbol $\kappa_{\beta}.$ This umbra has gf $f(\kappa_{\ms \beta},z)=e^z$ and
represents the sequence $\{1,1,\ldots\}.$ Therefore the Bell umbra is the umbral counterpart of a Poisson r.v. $\hbox{Po}(1)$ of
parameter $1$ with mgf $e^{e^z-1},$ whose Taylor expansion has coefficients equal to the Bell numbers \cite{Dinardoeurop}.

\begin{theorem} \label{mom1}
Cumulants of the generalized umbral sum are $\En[\C_{\ibs,\smallX}(\kappa_{\alpha})] = \En[( \kappa_{\alpha} \punt 
\kappabs_{\smallX})^{\ibs}]$ for all ${\ibs} \in {\mathbb N}_{\ms 0}^d,$ and represent the cumulant sequence of $\{\En[\C_{\ibs,\smallX}(\kappa_{\beta} \punt \kappa_{\alpha})]\}$ according to (\ref{(mgfcum)}).
\end{theorem}
\begin{proof}
The first result follows from Proposition \ref{7}, by comparing $h_{\ibs}$ in (\ref{(detsum1)}) with
$$
\En \left[\C_{\ibs, \smallX}(\kappa_{\alpha}) \right] = \ibs! \, \sum_{\lambdabs \vdash \ibs} \frac{\En[\kappa_{\alpha}^{\ms \lf(\lambdabs)}]} {\mf(\lambdabs)! \lambdabs!}
 \prod_{j} c_{\lambdabs_j}^{r_j}.
$$
For the second equality, according to Proposition \ref{7}, $f(\kappa_{\beta} \punt \kappa_{\alpha},z)=\exp \left( f(\kappa_{\alpha},z) - 1 \right)$ and the gf of
$\{\C_{\ibs,\Xbs}(\kappa_{\beta} \punt \kappa_{\alpha})\}$ is the composition of $\exp \left( f(\kappa_{\alpha},z) - 1 \right)$ and $f(\kappabs_{\smallX},\zbs).$
\end{proof}
\begin{corollary} \label{3.6}
$\En[\C_{\ibs,\smallX}(\kappa_{u})] = c_{\ibs}(\Xbs)$ and $\En[\C_{\ibs,\smallX}(\kappa_{\beta})]  = E[\Xbs^{\ibs}].$
\end{corollary}
\begin{example}[Common Clock Variance Gamma]
{\rm A common clock variance gamma model \cite{Deelstra} is a multivariate L\'evy stochastic process subordinated\footnote[5]{A subordinated stochastic process $\{Y_t\}$ is the stochastic time of a different stochastic process $\{X_T\},$ that is $T = Y_t$ a.s..} to a multivariate Brownian motion $\BBs_t$ with covariance rate $\Sigma t.$ Two sources of dependence are superimposed by using a univariate Gamma process
$G_t$ weighted by a $d$-dimensional drift $\thetabs.$ The resulting multivariate L\'evy process is
$\Xbs_t=\thetabs G_t + \BBs_{G_t}.$ Its mgf is
\begin{equation}
\phi_t(\zbs)=\left( \frac{1}{1- \nu \langle \thetabs, \zbs \rangle  - \frac{1}{2} \nu \langle \zbs, \zbs \Sigma \rangle } \right)^{t/\nu}
= \exp \left( - \frac{t}{\nu} \log \left[ 1 - \left( \nu \langle \thetabs, \zbs \rangle + \frac{1}{2} \nu \langle \zbs, \zbs \Sigma \rangle \right)\right] \right).
\label{(CCVG)}
\end{equation}
Consider the family $\{\kappa_{\alpha_t}\}$ with gf
$$f(\kappa_{\alpha_t},z) = 1 - \frac{t}{\nu} \log ( 1 -z)$$
whose coefficients are $\{\frac{t}{\nu} (i-1)!\}.$ Then the gf in (\ref{(CCVG)}) is of type (\ref{(mgfcum)}) with 
$\kappabs_{\tilde{\smallX}}$ given in Example \ref{Gauss} and $\tilde{\Xbs} \sim N(\thetabs \nu, \Sigma \nu).$
From Theorem \ref{mom1} we have $E[\Xbs_t^{\ibs}] = \En[\C_{\ibs,\tilde{\smallX}}(\kappa_{\beta} \punt \kappa_{\alpha_t})]$ with
$$f(\kappa_{\beta} \punt \kappa_{\alpha_t}, z) = \left(\frac{1}{1-z}\right)^{t/\nu} = \sum_{i \geq 0} p_i(t,\nu) \frac{z^i}{i!} \quad \hbox{and} \quad p_i(t,\nu) = \sum_{\lambda \vdash i} \binom{i}{r_1, r_2, \ldots} \left( \frac{t}{\nu} \right)_{\lf(\lambda)}$$
and $( \cdot )_{\lf(\lambda)}$ denotes the lower factorial.}
\end{example}
The homogeneity property of cumulants holds for cumulant polynomial sequences.
\begin{lemma}[Homogeneity] \label{homogeneity}
If $a \in {\mathbb C},$ then $\C_{\ibs,a \smallX}(y)=a^{|\ibs|} \C_{\ibs,\smallX}(y).$
\end{lemma}
\begin{lemma}[Convolution] \label{additivity}
If $\{ \Xbs_1, \ldots, \Xbs_n\}$ is a set of independent random vectors then
$$\C_{\ibs,\Xbs_1 + \cdots + \Xbs_n}(y) = \sum_{\mycom{(\ibs_1, \ldots,\ibs_n) \in {\mathbb N}_{\ms 0}^d}{\ibs_1+ \ldots+\ibs_n = \ibs}} \binom{\ibs}{\ibs_1, \ldots, \ibs_n} \C_{\ibs_1,\Xbs_1}(y) \cdots \C_{\ibs_n,\Xbs_n}(y).$$
\end{lemma}
\begin{proof}
Due to the additivity property of multivariate cumulants
\begin{equation}
\exp \left[ y K_{\smallX_1 + \cdots + \smallX_n} (\zbs) \right] = \exp \left[ y \sum_{i=1}^n K_{\smallX_i}(\zbs) \right] = \prod_{i=1}^n
\exp \left[ y K_{\smallX_i}(\zbs) \right].
\label{add1}
\end{equation}
The result follows by using the multinomial expansion.
\end{proof}
\begin{example}[Multivariate natural exponential models] \label{Example3.13} {\rm The symbolic method allows us to represent not only formal power series related to mgf's or cgf's  but also families of distributions functions, as for example natural exponential families (NEF). Indeed, let $\Theta$ be the largest open set in $\Real^d$ for which $E[e^{\langle \thetabs, \Xbs \rangle}] < \infty.$ The NEF generated by $\Xbs$ is
$$
F_{\smallX}(\thetabs) = \exp \left\{\langle \thetabs, \xbs \rangle - K_{\smallX}(\thetabs) \right\}, \qquad \thetabs \in \Theta,
$$
with $K_{\smallX}(\thetabs)$ the cgf of $\Xbs.$ Then $F_{\smallX}(\thetabs)$ admits an expansion in formal power series such that
\begin{equation}
F_{\smallX}(\thetabs) = \sum_{\ibs \geq 0} \C_{\ibs,\footnotesize{\xbs}-\smallX}(1) \frac{\thetabs^{\ibs}}{\ibs!} \in {\mathbb C}[[\thetabs]].
\label{(expmod)}
\end{equation}
Indeed due to the property of semi-invariance for translation we have $K_{\footnotesize{\xbs}-\smallX}(\thetabs) =
\langle \xbs, \thetabs \rangle - K_{\smallX}(\thetabs)$ and equation (\ref{(expmod)}) follows from Lemma \ref{3.3}.
Since
$$M_{\smallX}(\zbs) = E[e^{\langle \xbs, \zbs \rangle}] = \exp\left\{ K_{\smallX}(\zbs + \thetabs) - K_{\smallX}(\thetabs)\} \right\}$$
moments of NEF can be recovered from cumulant polynomial sequences as shown in the following. Extend coefficient-wise
the evaluation $\En$ to ${\mathbb C}[\A] [[\zbs]]$ in order to have
\begin{equation}
f(\kappabs, \zbs) = \En \left[ e^{\langle \kappabs, \zbs \rangle} \right] = \En \left[ e^{\langle \kappabs, (\zbs-\thetabs) \rangle}
e^{\langle \kappabs, \thetabs \rangle} \right] = f(\kappabs, \thetabs) + \sum_{i > 0} c_{\ibs,\thetabs} \frac{(\zbs - \thetabs)^{\ibs}}{\ibs!}
\label{NEF}
\end{equation}
where $c_{\ibs,\thetabs} = \En \left[\kappabs^{\ibs} e^{\langle \kappabs, \thetabs \rangle}\right].$ Equation (\ref{NEF}) is the Taylor expansion of $f(\kappabs, \zbs)$ with initial point $\thetabs$ so that we can define an umbra $\kappabs_{\thetabs}$ representing the sequence $\{c_{\ibs,\thetabs}\}$ such that $f(\kappabs_{\thetabs}, \zbs) = 1 + f(\kappabs, \zbs + \thetabs) - f(\kappabs, \thetabs).$ Choose $K(\zbs)$ as $f(\kappabs,\zbs)-1,$ then $\En \left[\C_{\ibs, \Xbs_{\thetabs}} (\kappa_{\beta}) \right] = E[\Xbs^{\ibs}].$ Some remarks on further developments of this approach
are added at the end of the paper.}
\end{example}
%
%----------------------------------------------------------------------------------------------------
\section{Multivariable generalizations}
%----------------------------------------------------------------------------------------------------
A first way to generalize cumulant polynomial sequences to the multivariate case is when a summation of indeterminates $y_1 + \cdots + y_n$ is plugged in $\C_{\ibs,\smallX}(y)$ and polynomials on the ring ${\mathbb C}[y_1, \ldots, y_n]$ are considered. A multinomial expansion of ${\mathcal C}_{\ibs,\smallX}(y_1 + \cdots + y_n)$ is given in the following theorem and the proof follows the same arguments given in Lemma \ref{additivity}.
\begin{theorem}[Multinomial property] \label{multinomial}
$$\C_{\ibs,\smallX}(y_1 + \cdots + y_n)=\sum_{\mycom{(\ibs_1, \ldots,\ibs_n) \in {\mathbb N}_{\ms 0}^d}{\ibs_1+ \ldots+\ibs_n = \ibs}} \binom{\ibs}{\ibs_1, \ldots, \ibs_n} \C_{\ibs_1,\smallX}(y_1) \cdots \C_{\ibs_n,\smallX}(y_n).$$
\end{theorem}
\begin{corollary}
Let ${\mathcal P}_{n}(\ibs)$ be a set of augmented matrices\footnote[6]{ An augmented matrix is a matrix obtained by appending the columns
of two or more given matrices.} such that ${\mathcal P}_{n}(\ibs) = \{\lambdabs=(\lambdabs_{\smalluno} | \ldots|\lambdabs_n):\lambdabs_{j} \vdash \ibs_{j} \,\, \hbox{for} \,\, j=1,\ldots,n \,\, \hbox{and} \,\, \ibs_{\smalluno} +  \cdots + \ibs_{\smalln} = \ibs \},$ then
\begin{equation}
{\mathcal C}_{\ibs,\smallX}(y_1 + \cdots + y_n) = \ibs! \sum_{\lambdabs \in {\mathcal P}_{\ms n}(\ibs)} \frac{y_{\smalluno}^{\ms \lf(\lambdabs_{\smalluno})} \cdots y_{n}^{\ms \lf(\lambdabs_n)}}{\mf(\lambdabs)! \lambdabs!}  \prod_{j} c_{\lambdabs_j}^{t_j}
\label{(detsum11bis)}
\end{equation}
where $t_j$ is the multiplicity of $\lambdabs_j$ in $\lambdabs.$ 
\end{corollary}
\begin{proof}
Let us consider two multi-index partitions $\lambdabs_1 \vdash \ibs_1$ and $\lambdabs_2 \vdash \ibs_2$ such that
$\lambdabs_1 = ( \lambdabs_{1,1}^{r_1}, \lambdabs_{1,2}^{r_2}, \ldots)$ and $\lambdabs_2 = (\lambdabs_{2,1}^{s_1},
\lambdabs_{2,2}^{s_2}, \ldots),$ then
$$\left( \prod_{j} c_{\lambdabs_{1,j}}^{r_j} \right) \left( \prod_{j} c_{\lambdabs_{2,j}}^{s_{j}} \right) = \prod_j c _{\lambdabs_j}^{t_j},$$
where $\lambdabs$ is the augmented matrix obtained by appending $(\lambdabs_1 | \lambdabs_2).$ Then equal columns are grouped in such a way that any $\lambdabs_j$ appears once and $t_j$ gives the number of column vectors equal to $\lambdabs_j$ in $\lambdabs.$ 
Moreover we have $\mf(\lambdabs)!=\mf(\lambdabs_1)!\mf(\lambdabs_2)!$ and $\lambdabs! =  \lambdabs_1!\lambdabs_2!.$
This construction can be extended to the multi-indexes $(\ibs_1, \ldots,\ibs_n)$ involved in the multinomial property \ref{multinomial}
giving the augmented matrix $(\lambdabs_1 | \cdots | \lambdabs_n)$ and the set ${\mathcal P}_{n}(\ibs).$ The result follows from the multinomial property \ref{multinomial} making the multiplications on the right-hand-side of
$$\C_{\ibs,\smallX}(y_1 + \cdots + y_n)= \sum_{\mycom{(\ibs_1, \ldots,\ibs_n) \in {\mathbb N}_{\ms 0}^d}{\ibs_1+ \ldots+\ibs_n = \ibs}} \ibs!\left( \sum_{\lambdabs_1 \vdash \ibs_1} \frac{y^{\ms \lf(\lambdabs_1)}}{\mf(\lambdabs_1)! \lambdabs_1!}
\prod_{j} c_{\lambdabs_{1,j}}^{r_j} \right) \cdots \left( \sum_{\lambdabs_n \vdash \ibs_n } \frac{y^{\ms \lf(\lambdabs_n)}}{\mf(\lambdabs_n)! \lambdabs_n!}
\prod_{j} c_{\lambdabs_{n,j}}^{s_j} \right).$$
\end{proof}
From Theorem \ref{mom1}, if we replace $\kappa_{\alpha}$ with $\langle \kappabs_{\beta}, \ybs \rangle,$ where
$\ybs = (y_1, \ldots, y_n)$ and $\kappabs_{\beta} = (\kappa_{\beta}, \ldots, \kappa_{\beta^{\prime}}),$ then
$\C_{\ibs,\smallX}(y_1 + \cdots + y_n)$ are cumulants of a multivariable generalized umbral sum.

\begin{corollary} ${\mathcal C}_{\ibs,\smallX}(y_1 + \cdots + y_n) = \En[(\langle \kappabs_{\beta}, \ybs \rangle \punt \kappabs_{\smallX})^{\ibs}].$
\end{corollary}
\begin{example}[Multivariate Compound Poisson r.v.]
{\rm If $a_i \in \Real$ is plugged in $y_i$ for $i=1,\ldots,n,$ then $\C_{\ibs,\smallX}(a_1 + \cdots + a_n)$ gives the $\ibs$-th multivariate cumulant of a random sum $\Sbs_{\ms N_{\ms 1}} + \cdots + \Sbs_{\ms N_{\ms n}},$ where $\Sbs_{\ms N_{\ms i}} = \Xbs_{\ms i,\smalluno} + \cdots + \Xbs_{\ms i, N_i}$ with $\{\Xbs_{\ms i,j}\}$ independent random vectors i.d. to $\Xbs,$
and with $N_{\ms i}(a_i)$ independent Poisson r.v.'s of parameter $a_i$ for $i=1, \ldots, n.$}
\end{example}
A further generalization of the multivariable cumulant polynomial ${\mathcal C}_{\ibs,\smallX}(y_1 + \cdots + y_n)$ is when r.v.'s (or umbrae) are plugged in $y_1, \ldots, y_n.$ If these r.v.'s are not independent we have
$$\En \left[ \C_{\ibs,\smallX}(Y_1 + \cdots + Y_n) \right] = \ibs! \, \sum_{\lambdabs \in {\mathcal P}_{\ms n}(\ibs)} \frac{c_{\lambdabs}(\Ybs)}{\mf(\lambdabs)! \lambdabs!} \prod_{j} [c_{\lambdabs_j}(\Xbs)]^{t_j} $$
with $c_{\lambdabs}(\Ybs)$ the joint cumulant of $\Ybs = (Y_1, \ldots, Y_n)$ of order $(\lf(\lambdabs_{\smalluno}), \cdots, \lf(\lambdabs_n)). $

\begin{example} [Photocounting]
{\rm Let $\Nbs = (N_1, \ldots, N_n)$ be a Poisson random vector with random intensity ${\boldsymbol I}=(I_1, \ldots, I_n).$ Then the $i$-th cumulant polynomial $\En[\C_{i, \langle \Nbs, \unobs \rangle}(I_1 + \cdots + I_n)]$ represents the $i$-th cumulant of a mixed Poisson distribution with random parameter $I_1 + \cdots + I_n.$ When ${\boldsymbol I}$ is the diagonal of a Wishart random matrix then $\langle \Nbs, \unobs \rangle$ is employed in photocounting \cite{DiNardophotocounting} and gives the number of electrons ejected by
$n$ pixels hit by a certain number of light waves.}
\end{example}
Let us consider again Theorem \ref{multinomial} and assume to choose different independent random $d$-tuples $\Xbs_1, \ldots, \Xbs_n$ as indexes of cumulant polynomials on the right-hand-side of (\ref{multinomial}). These new polynomials are said multivariable cumulant polynomials.
\begin{definition} \label{4.6}
The $\ibs$-th multivariable cumulant polynomial is
\begin{equation}
\C_{\ibs,(\Xbs_1,\ldots,\Xbs_n)}(y_1, \ldots, y_n) = \sum_{\mycom{(\ibs_1, \ldots,\ibs_n) \in {\mathbb N}_{\ms 0}^d}{\ibs_1+ \ldots+\ibs_n = \ibs}} \binom{\ibs}{\ibs_1, \ldots, \ibs_n} \C_{\ibs_1,\smallX_1}(y_1) \cdots \C_{\ibs_n,\smallX_n}(y_n).
\label{(multcum3)}
\end{equation}
\end{definition}
\begin{theorem}
The gf of the multivariable cumulant polynomial sequence is
$$\sum_{\ibs \geq 0}  \C_{\ibs,(\Xbs_1,\ldots,\Xbs_n)}(y_1, \ldots, y_n) \frac{\zbs^{\ibs}}{\ibs!} = \exp \left( \sum_{j=1}^n y_j K_{\smallX_j}(\zbs)\right).$$
\end{theorem}
\begin{proof}
The result follows from Lemma \ref{3.3} since the right-hand-side of (\ref{(multcum3)}) is the $\ibs$-th coefficient of the following product
\cite{FlS}
$$\left( \sum_{\ibs_1 \geq 0} \C_{\ibs_1, \smallX_1}(y_1) \frac{\zbs^{\ibs_1}}{\ibs_1!} \right) \cdots \left( \sum_{\ibs_n \geq 0} \C_{\ibs_n, \smallX_n}(y_n) \frac{\zbs^{\ibs_n}}{\ibs_n!} \right).$$
\end{proof}
\begin{example}[Multivariate Hermite polynomials]
{\rm Let us consider multivariate Hermite polynomials $H_{\ibs}(\ybs, \Sigma),$ orthogonal with respect to the multivariate Gaussian density with ${\mathbf 0}$ mean and covariance matrix $\Sigma$ of the full rank $d.$ The gf of $\{H_{\ibs}(\ybs, \Sigma)\}$ is
$$\sum_{\ibs \geq 0}  H_{\ibs}(\ybs, \Sigma) \frac{\zbs^{\ibs}}{\ibs!} = \exp \left\{ \langle \ybs, \zbs \rangle - \frac{1}{2} \langle \zbs, \zbs \Sigma \rangle \right\}.$$
Let us consider the following two random vectors. The first $\Ybs = (Y_{1}, \ldots, Y_n)$ built with r.v.'s i.i.d. to a r.v. $Y$ with $P(Y=y)=1$ so that $K_{\Ybs}(\zbs) = \langle \ybs, \zbs \rangle.$ The second is $\Xbs \simeq N({\mathbf 0}, \Sigma)$ with $\frac{1}{2} \langle \zbs, \zbs \Sigma \rangle = K_{\Xbs}(\zbs)$ according to Example \ref{Gauss}. Then we have
$$H_{\ibs}(\ybs, \Sigma) = \C_{\ibs,(\Ybs,\Xbs)}(1,-1).$$}
\end{example}
If correlated r.v.'s $Y_1, \ldots, Y_n$ are plugged in the indeterminates $y_1, \ldots, y_n,$ then
\begin{equation}
\sum_{\ibs \geq 0}  E \left[\C_{\ibs,(\Xbs_1,\ldots,\Xbs_n)}(Y_1, \ldots, Y_n) | \Ybs\right] \frac{\zbs^{\ibs}}{\ibs!} =
\exp \left( \sum_{j=1}^n Y_j K_{\smallX_j}(\zbs)\right),
\label{corrcond}
\end{equation}
where $E[ \, \cdot \, | \Ybs]$ is the conditional mean with respect to $\Ybs.$
\begin{corollary} \label{corrcond1} If $K_{\smallY}(\zbs)$ is the cgf of $\Ybs$ and $\{K_{\smallX_{\ms \smalluno}}(\zbs), \ldots, K_{\smallX_{\ms n}}(\zbs)\}$
are cgf's of $\Xbs_1,\ldots,\Xbs_n$ respectively, then
\begin{equation}
\sum_{\ibs > 0} \En \left[ \C_{\ibs,(\Xbs_1,\ldots,\Xbs_n)}(\kappa_{\smalluno}, \ldots, \kappa_{n})\right] \frac{\zbs^{\ibs}}{\ibs!} =
K_{\ms \smallY}\left[ K_{\smallX_{\ms \smalluno}}(\zbs), \ldots, K_{\smallX_{\ms n}}(\zbs) \right],
\label{(faamult1)}
\end{equation}
with $\kappabs_{\smallY} = (\kappa_{\smalluno}, \ldots,\kappa_{n})$ the $d$-tuple representing joint cumulants of $\Ybs.$
\end{corollary}
Corollary \ref{corrcond1} gives a way to compute the multivariate Fa\`a di Bruno's formula. In particular the composition in the right hand side of
(\ref{(faamult1)}) corresponds to the composition (\ref{(faamult)}) with $h$ replaced by $K_{\ms \smallY}$ and $f_{\ms i}$ replaced by $K_{\smallX_{\ms i}}.$ Since the multivariate Fa\`a di Bruno's formula gives the coefficients of multivariate compositions, according to Corollary \ref{corrcond1} these coefficients correspond to $\C_{\ibs,(\Xbs_1,\ldots,\Xbs_n)}(\kappa_{\smalluno}, \ldots, \kappa_{n})$ that can be computed by using Definition \ref{4.6} and Lemma \ref{multcomp}. Note that Definition \ref{4.6} highlights the additive structure of the formula.  
%-----------------------------------------------------------------
\subsection{Sampling and symmetric polynomials}
%-----------------------------------------------------------------
%
The polynomials $\{C_{\ibs,\Xbs}(y_1 + \cdots + y_n)\}$ are symmetric in $\{y_1, \ldots, y_n\}.$
Different symmetric polynomials can be obtained as follows. For lightning the notation, we assume $d=1.$
If $\{X_1, \ldots, X_n\}$ are independent r.v.'s then
$$\En[(\kappa_{{\ms X_1} y_1} + \cdots + \kappa_{{\ms X_n} y_n})^i] = y_1^i \En[\kappa^i_{{\ms X_1}}] + \cdots + y_n^i \En[\kappa^i_{{\ms X_n}}]$$
from the additivity property of cumulants. Moreover, if $\{X_1, \ldots, X_n\}$ are also i.d. to $X$ then 
\begin{equation}
\En[(\kappa_{{\ms X_1} y_1} + \cdots + \kappa_{{\ms X_n} y_n})^i] = \En[\kappa^i_{{\ms X}}] s_i
\label{(sampling)}
\end{equation}
with $s_i$ the $i$-th symmetric power sum in the indeterminates $\{y_1, \ldots, y_n\},$ that is $s_i = \sum_{j=1}^n y_j^i.$
In particular, we have
$f\left(\kappa_{{\ms X_1} y_1} + \cdots + \kappa_{{\ms X_n} y_n}, z \right) = \exp \left\{ \sum_{j=1}^n [f(y_j \kappa_{\ms X}, z) - 1] \right\}.$
Sequences represented by $\kappa_{{\ms X_1} y_1} + \cdots + \kappa_{{\ms X_n} y_n}$ give sample statistics in $\{X_1, \ldots, X_n\},$
as shown in the following corollary. Its proof follows from Corollary \ref{3.6}, using (\ref{(detsum)}) and (\ref{(sampling)}).
\begin{corollary} \label{111}
\begin{equation}
C_{i,\langle \Xbs, \ybs \rangle}(1) = E[(X_1 y_1 + \ldots + X_n y_n)^i] = \sum_{\lambda \vdash i} \frac{i!}{(1!)^{r_1} r_1! (2!)^{r_2} r_2! \cdots} \prod_{j} \left[ c_j(X) s_{j} \right]^{r_j}.
\label{(sampling2)}
\end{equation}
\end{corollary}
Thanks to (\ref{(sampling)}), power sum symmetric polynomials are special cumulant polynomial sequences involving cumulants of r.v.'s. Next example shows that elementary symmetric polynomials are special cumulant polynomial sequences too, having a representation similar to the left hand side of (\ref{(sampling)}), but depending on umbrae that do not have a counterpart within r.v.'s. 
\begin{example}[Elementary symmetric polynomials]
{\rm Elementary symmetric polynomials in the indeterminates $\{y_1, \ldots, y_n\}$
$$e_i(y_1, \ldots, y_n) = \sum_{1 \leq j_1 < j_2 < \cdots < j_i \leq n} y_{j_1} y_{j_2} \cdots y_{j_i}$$ are special cumulant polynomial sequences. Indeed we have
\cite{FlS}
$$\sum_{i \geq 0} e_i(y_1, \ldots, y_n) \frac{z^i}{i!} = \prod_{j=1}^n (1 + y_j z) = \exp \left\{ \sum_{j=1}^n \log(1+y_j z) \right\}.$$
Observe that  $1 + \log(1+z)$ is the compositional inverse of $e^{z} = f(\kappa_{\beta },z).$ Denote by  $\kappa_{\beta^{\ms <-1>}}$
the umbral monomial such that $f(\kappa_{\beta^{\ms <-1>}},z)=1 + \log(1 + z),$ then
$$\sum_{i \geq 0} e_i(y_1, \ldots, y_n) \frac{z^i}{i!} =  \prod_{j=1}^n \exp \Bigl\{ f\left(y_j \kappa_{\beta^{\ms <-1>}},z\right) - 1  \Bigr\}.$$
Having gf as (\ref{add1}), elementary symmetric polynomials are cumulant polynomial sequences. Then the sequence $\{e_i(y_1, \ldots, y_n)\}$ is umbrally represented by 
$\kappa_{y_1 \beta^{\ms <-1>}} + \cdots + \kappa_{y_n \beta^{\ms <-1>}},$ since \cite{FlS}
$$\prod_{j=1}^n (1 + y_j z) = \exp \left\{ \sum_{i > 0} (-1)^{i-1} (i-1)! s_i \frac{z^i}{i!} \right\}.$$
Note that the second cumulant of $\beta^{\ms <-1>}$ is negative.}
\end{example}

Equation (\ref{(sampling2)}) allows us to recover the definition of cumulants of random matrices, as given in \cite{Peter} by using  different arguments.

\begin{definition} \label{rm}
If $\left\{X_1, \ldots, X_n\right\}$ are correlated r.v.'s representing the eigenvalues of a random matrix $A$, then
the sequence of cumulants $\{c_i(A)\}$ of $A$ is such that
\begin{equation} 
E\left\{ \left[ \Tr(A) \right]^i \right\} = C_{i, A}(n) = i! \sum_{\lambda \vdash i} \frac{n^{\lf(\lambda)}}{(1!)^{r_1} r_1! (2!)^{r_2} r_2! \cdots}\prod_{j} \left[ c_j(A) \right]^{r_j},
\label{(4.8bis)}
\end{equation}
where $C_{i, A}(\cdot)$ is the $i$-th cumulant polynomial (\ref{(detsum)}).
\end{definition}
Let us motivate Definition \ref{rm}. In (\ref{(sampling2)}), assume to replace $y_1, \ldots, y_n$ with $1.$ Then we have
\begin{equation}
E[(X_1 + \ldots + X_n )^i] = i! \sum_{\lambda \vdash i} \frac{n^{\lf(\lambda)}}{(1!)^{r_1} r_1! (2!)^{r_2} r_2! \cdots}
\prod_{j} \left[ c_j(X) \right]^{r_j}.
\label{(sampling3)}
\end{equation}
If $\left\{X_1, \ldots, X_n\right\}$ are correlated r.v.'s representing the eigenvalues of a random matrix $A$,
then $E[(X_1 + \ldots + X_n )^i] = E\left[ \Tr(A)^i \right].$ If we compare (\ref{(4.8bis)}) with (\ref{(detsum)}), from Lemma 3.2 the gf of $\{C_{i, A}(n)\}$ is $\exp [n \, K_{\smallX}(\zbs)] = M_{\smallX}(\zbs)$ with $M_{\smallX}(\zbs)$ the mgf of $\Xbs=(X_1, \ldots, X_n )$ and $\zbs = (z,\ldots,z).$  Then Definition \ref{rm} states that cumulants of $A$ are the cumulant sequence of the coefficients in the Taylor expansion of $M_{\smallX}(\zbs).$ 
\begin{corollary}
The $i$-th cumulant of $A$ is $c_i(A) = C_{i, A}  \left( \frac{1}{n}\right).$
\end{corollary}
\begin{proof}
The result follows as $K_{\smallX}(\zbs)$ is the composition of $1 + \frac{z}{n}$ and $\log M_{\smallX}(\zbs).$
\end{proof}
By using Definition \ref{rm} of cumulants of a random matrix, we prove that cumulants of simple random samples do not depend from the size. This property explains why $k$-statistics\footnote[7]{$k$-statistics are unbiased estimators of cumulants.} have the property to be natural statistics\footnote[8]{A statistic $T$ is said to be {\it natural} if for each $m \leq n$ the average value of $T_m( \cdot )$ over random samples $(Y_1, \ldots, Y_m)$ drawn from $(X_1, \ldots, X_n)$ is equal to $T_n(X_1, \ldots, X_n).$} \cite{Peter}.

Simple random sampling may be performed as follows: consider the matrix $A_X = \hbox{\rm diag}(X_1,\ldots,X_n)$ of i.i.d.r.v.'s and an orthogonal matrix $P$ with the $(i,j)$-entry $(P)_{i,j} = \delta_{i,\sigma(i)},$ according to a given permutation $\sigma \in {\mathcal S}_n,$ with ${\mathcal S}_n$ the symmetric group and $\delta_{i,j}$ the Kroneker delta. Assume to delete the last $n-m$ rows in $P.$ Then $P_{n-m}$ is a rectangular matrix such that $P_{n-m} P^{\trasp}_{n-m} = I_m$ and $P^{\trasp}_{n-m} P_{n-m} \ne I_n.$ The random matrix $A_{Y}$ such that $A_Y = P_{n-m} A_X P^{\trasp}_{n-m}$ is a diagonal matrix, $A_Y = {\rm diag}(Y_1, \ldots, Y_m),$ whose elements correspond to a simple random sample drawn from $(X_1, \ldots, X_n).$
\begin{theorem}[Simple random sampling] If $m \leq n$ then $\C_{i, A_Y} \left(\frac{1}{m}\right) = 
\C_{i, A_X} \left(\frac{1}{n}\right).$
\end{theorem}
\begin{proof}
Since the trace is invariant under permutations, first observe that $\left[ \Tr(A_Y) \right]^i  = \left[ \Tr(A_X \tilde{P}) \right]^i$ with $\tilde{P} = P^{\trasp}_{n-m} P_{n-m}.$ From (\ref{(4.8bis)}), we have $\C_{i, A_Y} \left(\frac{1}{m}\right)  =  \C_{i, A_X \tilde{P}} \left(\frac{1}{m}\right)$ and
\begin{equation}
\sum_{i \geq 0} \C_{i, A_X \tilde{P}} \left(\frac{1}{m}\right) \frac{z^i}{i!} = \exp\left\{ \frac{1}{m} K_{A_X \tilde{P}}(z) \right\}.
\label{(eq1)}
\end{equation}
Let us observe that $c_i(A_X \tilde{P}) = c_i(A_X) m^i$ from (\ref{(sampling)}), because of $\Tr(\tilde{P})=m.$ Therefore
$K_{A_X \tilde{P}} (z) = K_{A_X} (m \, z) = \frac{1}{n} \log M_{A_X}(m \, z).$ Since $X_1, \ldots, X_n$ are i.i.d.r.v.'s then
$M_{A_X}(m \, z) = \left( M_{A_X} (z) \right)^m$ and $K_{A_X \tilde{P}}(z) = \frac{m}{n} K_{A_X}(z).$
The result follows by replacing the last equality in (\ref{(eq1)}).
\end{proof}
%---------------------------------------------------------------------------------
\section{Conclusions}
%---------------------------------------------------------------------------------
%
In this paper, we have introduced a new family of polynomials, the cumulant polynomial sequence, and proved that they are a very adaptable tool having applications in probability and statistics and allowing us to deal with both cumulants and moments. 

The connection with exponential models (Example \ref{Example3.13}) opens new fields where cumulant polynomials can be fruitfully applied, in dealing with distribution functions instead of number sequences. Some steps in this direction was made in \cite{DiBucchianico}, but since then the umbral theory has been enriched with many tools and more can be done not only by using the approach here proposed but also extending the analysis to multivariable Sheffer polynomial sequences, which are cumulant polynomial sequences too. 
 
Indeed, formal power series as $F_{\smallX}(\thetabs)$ in (\ref{(expmod)}) involve multivariable Sheffer polynomial sequences
as shown in the following. Keeping the notation with the parameter $\thetabs,$ a multivariable Sheffer polynomial sequence $\{s_{\ibs}(\thetabs)\}$ has gf of type  
\begin{equation}
\sum_{\ibs \geq 0} s_{\ibs}(\thetabs) \frac{\thetabs^{\ibs}}{\ibs!} = g(\thetabs) \exp \left\{ K(\thetabs)\right\} \in {\mathbb C}[[\thetabs]].
\label{(multshef)}
\end{equation}
In \cite{Brown}, a theory of multivariable Sheffer polynomial sequences is introduced when $g( \cdot )$ is a univariate gf.
Since in the ring ${\mathbb C}[[\thetabs]],$ there exists a formal power series $\tilde{K}(\thetabs)$ such that $g(\thetabs) = \exp
\left\{ \tilde{K}(\thetabs) \right\},$ then multivariable Sheffer polynomial sequences $\{ s_{\ibs}(\thetabs)\}$ are
special multivariable cumulant polynomial sequences, having gf of type $\exp \left\{  \tilde{K}(\thetabs) + K(\thetabs)\right\}.$
An interesting problem, which deserves further developments, is to investigate NEF models for which the variance function has a particular form taking advantage of Sheffer expansion (\ref{(multshef)}). In Example \ref{Example3.13} recall that if $\mbs = K_{\smallX}^{\prime}(\thetabs)$ is the mean vector and $\psi(\mbs)$ is its inverse, the variance function is the composition $K_{\smallX}^{\prime \prime}\left( \psi(\mbs) \right)$ with $K_{\smallX}^{\prime \prime}$  the Hessian matrix of $K_{\smallX}(\thetabs).$ A connection between variance functions and Sheffer polynomial sequences has been investigated in \cite{DiBucchianico} for the univariate case. Its generalization to the multivariate case is still an open problem and we believe that multivariable cumulant polynomial sequences can be fruitfully employed to carry out such characterizations.
%
%---------------------------------------------------------------------------------
\section{Acknowledgements}
%---------------------------------------------------------------------------------
The author thanks Peter McCullagh for his useful comments and remarks in dealing with the symbolic method by using cumulants instead of moments. The author thanks also the referees for the suggestions which allow us to improve the overall presentation of the paper.

\end{document}